\documentclass{amsart}
\usepackage[utf8]{inputenc}
\usepackage{enumerate}

\usepackage{color}
\definecolor{Darkgreen}{rgb}{0,0.4,0}
\usepackage[colorlinks,linkcolor=Darkgreen,citecolor=Darkgreen,breaklinks=true]{hyperref}

\makeindex

\makeatletter

\def\laweq{\overset{\mathrm{law}}=}

\def\ind{\mathop{\rm ind}\nolimits}
\def\dotcup{\mathbin{\dot\cup}}

\def\d{\mathrm{d}}

\def\bbone{\boldsymbol 1 }
\def\<{\langle}
\def\>{\rangle}

\def\mindex#1{\emph{#1}}

\theoremstyle{plain}
\newtheorem{theorem}{Theorem}[section]
\newtheorem{lemma}[theorem]{Lemma}
\newtheorem{corollary}[theorem]{Corollary}
\newtheorem{proposition}[theorem]{Proposition}

\theoremstyle{definition}
\newtheorem{definition}[theorem]{Definition}
\newtheorem{remark}[theorem]{Remark}
\newtheorem{example}[theorem]{Example}

\theoremstyle{remark}

\makeatother

\title{Markovian dynamics of exchangeable arrays}

  \author{Jiří Černý}
  \address{
    Department of Mathematics and Computer Science\\
    University of Basel\\
    Spiegelgasse 1\\
    4051 Basel, Switzerland\\
    jiri.cerny@unibas.ch}

  \author{Anton Klimovsky}
  \address{
    Fakultät für Mathematik\\
    Universität Duisburg-Essen\\
    Thea-Leymann-Str.~9\\
    45117 Essen, Germany\\
    anton.klymovskiy@uni-due.de}

\begin{document}

\begin{abstract}
  We study Markov processes with values in the space of general
  two-dimensional arrays whose distribution is exchangeable. The results
  of this paper are inspired by the theory of exchangeable dynamical
  random graphs developed by H.~Crane in \cite{Cra16,Cra17}.
\end{abstract}

\maketitle

\section{Introduction}
\label{s:intro}

The goal of this paper is to survey some of the recent results on the
Markovian dynamics of exchangeable random graphs due to Harry Crane
\cite{Cra16,Cra17} and generalise them to the context of dynamics
of exchangeable random arrays whose entries do not necessarily take
values in a finite set.

The paper extends the results presented by the authors in the learning
session ``Genealogies of particles on dynamic random networks'' during
the Programme ``Genealogies of Interacting Particle Systems'' of Institute
for Mathematical Sciences in August 2017. The learning session
concentrated on various aspects of the dynamics of random graphs and,
in particular, of particle systems on such graphs. While the original
theory due to H.~Crane cannot be applied directly in this context,  the
results of this paper could be relevant, e.g., for exchangeable Markovian
dynamics of particle systems on weighted exchangeable dynamical random
graphs.

Our results closely follow \cite{Cra16,Cra17}. However, as we cannot use
the fact that the entries of the array take values in a finite space,
some of the proofs require non-trivial modifications, which, in our
opinion, sometimes make them cleaner. Many of the results of
\cite{Cra16,Cra17} depend very strongly on the finiteness of the entry
space and cannot be proved easily in the general context. Those are
omitted here.

\smallskip\noindent\textbf{Acknowledgements.} The work was done partially while the
authors were visiting the Institute for Mathematical Sciences, National
University of Singapore in 2017. The visit was supported by the
Institute. AK's visit was also partially supported by the German Academic
Exchange Service (DAAD).

\section{Exchangeable random arrays}
\label{s:excharrays}

We will consider arrays with values in an arbitrary  Polish space $S$.
This space will be endowed with its Borel $\sigma $-field $\mathcal B(S)$
and a compatible metric $d_S$, which we assume to be bounded by $1$. We
write $\mathcal P(S)$ for the set of all probability measures on
$(S,\mathcal B(S))$ endowed with the topology of weak convergence, which
is a Polish space as well.

A random $S$-valued array is a collection $Y=(Y_{ij})_{ij\in\mathbb N}$
of $S$-valued random variables on some probability space
$(\Omega ,\mathcal A, P)$. Otherwise said, $Y$ is
$\mathbb S:=S^{\mathbb N^2}$-valued random variable. We endow $\mathbb S$
with the product topology and the compatible metric
$d_{\mathbb S}(y,y')=\sum_{i,j\in \mathbb N} 2^{-i-j} d_S(y_{ij},y'_{ij})$.

For an arbitrary set $A\subset \mathbb N$, we define
$Y|_{A}= (Y_{ij})_{i,j\in A}$ to be the restriction of $Y$ to the index set
$A$. In particular, with $[n]:=\{1,\dots,n\}$, $Y|_{[n]}$ is its
restriction to the first $n$ coordinates,
$Y|_{[n]}\in \mathbb S_n:=S^{n^2}$.

Similarly, for every probability distribution $\nu$ on $\mathbb S$ (or on
  $\mathbb S_m$, $m\ge n$), we denote by $\nu|_{[n]}$ its image under the
canonical restriction from $\mathbb S$ (or $\mathbb S_m$) to $\mathbb S_n$.
It is a known fact
that a sequence of probability measures $(\mu^k)_{k\ge 1}$ on  $\mathbb S$
converges weakly to $\mu \in \mathcal P(\mathbb S)$, iff all restrictions
$\mu^k|_{[n]}\in \mathcal P(\mathbb S_n)$,  converge weakly in
$\mathcal P(\mathbb S_n)$, or equivalently $\mu^k(f)\to \mu(f)$, for
every bounded continuous cylinder function $f$ on $\mathbb S$.

Let $\Sigma $ be the set of all permutations of integers, that is the set
of all bijections of $\mathbb N$ which fix all but finitely many values;
$\Sigma_n$ denotes the set of all permutations of $[n]$. For an  array $Y$
and $\boldsymbol \pi = (\pi_1 ,\pi_2)\in \Sigma^2 $, we define a new array
$Y^{\boldsymbol \pi }$ by
$Y^{\boldsymbol \pi}_{ij} = Y_{\pi_1(i)\pi_2(j)}$. For $\pi \in \Sigma $,
we also define $Y^\pi$ by $Y^\pi_{ij}= Y_{\pi(i) \pi(j)}$. An
array $Y$ is called \mindex{exchangeable} if
\begin{equation}
  \label{e:exar}
  Y \laweq Y^{\boldsymbol \pi},
  \qquad \text{for every $\boldsymbol \pi \in \Sigma^2$}.
\end{equation}
An array $Y$ is called \mindex{weakly exchangeable}%
\footnote{The terminology is slightly misleading:
  due to the symmetry requirement, the weak exchangeability is not  weaker
  than the exchangeability}
if it is symmetric (i.e., $Y_{ij}=Y_{ji}$)
and
\begin{equation}
  \label{e:wexar}
  Y \laweq Y^{\pi}, \qquad \text{for every $\pi \in \Sigma$.}
\end{equation}

The key result of the theory of random arrays is their characterisation
due to Aldous \cite{Ald81} and Hoover \cite{Hoo79} which can be viewed as
a ``two-dimensional version'' of de Finetti's theorem.

\begin{theorem}
  \label{t:AH}
  \textup{(a)} If $(Y_{ij})_{i,j\in \mathbb N}$ is an $S$-valued exchangeable
  array, then there exists a measurable function $f\colon[0,1]^4\to S$ such
  that $Y\laweq Y^\star$, where
  \begin{equation}
    Y^\star_{ij}= f(U, U_i, U'_j, U_{ij}),
  \end{equation}
  and $U, (U_i)_{i\in \mathbb N}$, $(U'_i)_{i\in \mathbb N}$, and
  $(U_{ij})_{i,j\in \mathbb N}$ are independent collections of
  $\mathrm{Uniform}([0,1])$ i.i.d.~random variables.

  \textup{(b)} If $(Y_{ij})_{i,j\in \mathbb N}$ is an $S$-valued weakly
  exchangeable array, then the analogous statement holds with a function
  $f\colon[0,1]^4\to S$ satisfying $f(\cdot,x,y,\cdot)=f(\cdot,y,x,\cdot)$, and
  with
  \begin{equation}
    Y^\star_{ij} = Y^\star _{ji}= f(U, U_i, U_j, U_{ij}), \qquad i\ge j.
  \end{equation}
\end{theorem}

The representing function $f$ of the Aldous-Hoover theorem is not uniquely
determined.
E.g., in the case (a), if two functions $f$ and $f'$ satisfy
$f'(a,b,c,d)=f(T_1(a),T_2(b),T_3(c),T_4(d))$ for some measure preserving
transformations $T_1,\dots,T_4$ of $[0,1]$, then the corresponding
exchangeable arrays have the same distribution.

A (weakly) exchangeable array is called \mindex{dissociated} if
\begin{equation}
  (Y_{ij}:i,j\le n) \text{ is independent of }
  (Y_{ij}:i,j>n), \text{ for each $n$.}
\end{equation}
It is obvious that if the function $f$ in the representation of
Theorem~\ref{t:AH} does not depend on the first coordinate, then $Y$ is
dissociated. Converse statement hold as well, see
Corollary~14.13 in \cite{Ald85}.

Dissociated arrays play a similar role as i.i.d.~sequences do in the
theory of exchangeable sequences: Every (weakly) exchangeable array is a
mixture of (weakly) exchangeable dissociated arrays. To state this more
formally, we need more definitions.

A set $A\in \mathcal B( \mathbb S)$ is called \textit{exchangeable}
\index{exchangeable set} if $A=A^{\boldsymbol \pi}$ for every
$\boldsymbol \pi\in \Sigma^2 $, where
$A^{\boldsymbol \pi}=\{y^{\boldsymbol \pi}:y\in A\}$ and
$y^{\boldsymbol \pi}_{ij} = y_{\pi_1(i)\pi_2(j)}$. The collection
$\mathcal E_S\subset \mathcal B(\mathbb S)$ of all exchangeable sets is
called the \mindex{exchangeable $\sigma$-field}. For an exchangeable
array $Y$, we define
$\mathcal E_Y=\{Y^{-1}(A): A\in \mathcal E_S\}$. We use
$\mathcal D_S \subset \mathcal P(\mathbb S)$ to denote the set of all
distributions of dissociated exchangeable arrays, which is a closed
subset of $\mathcal P(\mathbb S)$. We write $\tilde{\mathcal D}_S$ for
the set of all distributions of dissociated weakly exchangeable arrays.

The following proposition follows from~Proposition 14.8 and Theorem~12.10
of \cite{Ald85}.

\begin{proposition}
  \label{p:mixture}
  \textup{(a)} A (weakly) exchangeable array $Y$ is dissociated iff $P(A)\in \{0,1\}$
  for every $A\in \mathcal E_Y$, that is its exchangeable $\sigma $-field
  is $P$-trivial.

  \textup{(b)} Let $Y$ be a (weakly) exchangeable array and $\alpha $ its regular
  conditional distribution given $\mathcal E_Y$. Then,
  $\alpha (\omega )\in \mathcal D_S$
  (resp.~$\alpha (\omega )\in \tilde{\mathcal D}_S$) for $P$-a.e.~$\omega $. Moreover,
  the distribution
  $\mu_Y$ of $Y$ can be written as
  \begin{equation}
    \label{eq:mixture}
    \mu_Y(\cdot)=\int_{\mathcal D_S}\nu(\cdot) \Lambda_Y( \d \nu )
  \end{equation}
  for a uniquely determined probability measure $\Lambda_Y$ on
  $\mathcal D_S$ (resp.~$\tilde{\mathcal D}_S$).
\end{proposition}

An important feature of exchangeable arrays is that
regular conditional distribution $\alpha $ of $Y$ given
$\mathcal E_Y$ can, a.s., be recovered from a realisation of $Y$ by the following
procedure. For $m\ge n$,
and $y\in \mathbb S$,
let $t^{y,n}_m \in \mathcal P(\mathbb S_n)$ be defined by
\begin{equation}
  t^{y,n}_m = \frac 1 {((m)_n)^2} \sum_{\psi_1,\psi_2}
  \delta_{(y_{\psi_1(i),\psi_2(j)})_{i,j\in[n]}},
\end{equation}
where the sum runs over all injections $\psi_1,\psi_2\colon [n]\to[m]$ and
\begin{align}
(m)_n = m {(m-1)}\dots (m-n+1).
\end{align}
Measure $t^{y,n}_m$ can be viewed as the empirical
distribution of $n\times n$ sub-arrays in the array $y|_{[m]}$.
We further define
\begin{equation}
  \label{e:tn}
  t^{y,n} = \lim_{m\to \infty} t^{y,n}_m
\end{equation}
whenever this limit exists in the weak sense, and set
$|y|=(t^{y,n})_{n\ge 1}$  whenever all $t^{y,n}$, $n\ge 1$, exist.

It follows from the construction that the probability
measures $t^{y,n}_m$, $n=1,\dots,m$, are consistent in the sense that
$t^{y,n}_m|_{[n-1]} = t^{y,n-1}_m$ for every $2\le n\le m$. This
consistence transfers to the limit, that is
\begin{equation}
  \label{e:consistency}
  t^{y,n}|_{[n-1]} = t^{y,n-1}, \qquad \text{for every $n\ge 2$. }
\end{equation}
Therefore, in view of Kolmogorov's extension theorem,
$|y|$, when it exists, can be viewed as an element of $\mathcal P(\mathbb S)$.

In the weakly exchangeable case, we set
$\tilde t^{y,n}_m$ by
\begin{equation}
  \label{e:tnweak}
  \tilde t^{y,n}_m = \frac 1 {(m)_n} \sum_{\psi}
  \delta_{(y_{\psi(i),\psi(j)})_{i,j\in[n]}},
\end{equation}
where the sum runs over all injections $\pi $ from $[n]$ to $[m]$.
We then define $\tilde t^{y,n}$ and $|y| = (\tilde t^{y,n})_{n\ge 1}$
analogously as in the exchangeable case.

The next proposition establishes the connection between $|Y|$ and its
regular conditional distribution $\alpha$.

\begin{proposition}
  \label{p:dislimits}
  \textup{(a)} Let $Y$ be a (weakly) exchangeable array and $\alpha $ its regular
  conditional distribution given $\mathcal E_Y$. Then, for
  $P$-a.e.~$\omega $, $|Y(\omega )|$ exists and equals to
  $\alpha (\omega )$.
  In particular, $|Y(\omega )|\in \mathcal D_S$
  (resp.~$|Y(\omega )|\in \tilde {\mathcal D}_S$), $P$-a.s.

  \textup{(b)} If $Y$ is dissociated, then $|Y|$ exists a.s.~and
  coincides with the distribution of $Y$.
\end{proposition}

\begin{proof}
  We only sketch the argument. We assume first that the law of $Y$ is
  dissociated and denote it by $\alpha \in \mathcal D_S$. In this case,
  we should show that, $P$-a.s., for all $n\ge 1$,
  $\lim_{m\to\infty}t^{Y,n}_m = \alpha |_{[n]}$ weakly. Let
  $f\colon \mathbb S_n \to \mathbb R$ be a bounded continuous function.
  Then,
  \begin{equation}
    t^{Y,n}_m(f)= \frac 1 {((m)_n)^2} \sum_{\psi_1,\psi_2}
    f\big((Y_{\psi_1(i),\psi_2(j)})_{i,j\in[n]}\big).
  \end{equation}
  Using the fact that $Y$ is dissociated and thus $Y_{ij}$ is independent
  of $Y_{i'j'}$ when $i\neq i'$ and $j\neq j'$, it is then
  straightforward to extend the usual law-of-large-numbers type arguments
  to show that $\lim_{m\to \infty}t_m^{Y,n}(f) = \alpha |_{[n]}(f)$, $P$-a.s.
  To show that this convergence holds $P$-a.s.~jointly for all such $f$,
  one can then adapt the proof of Varadarajan's Theorem (see, e.g.,
    Theorem~11.4.1 in~\cite{Dud02}), which completes the proof in the
  dissociated case.

  In the general case, denoting by $\alpha $ the regular conditional
  distribution of $Y$ given $\mathcal E_Y$, using that
  $\alpha \in \mathcal D_S$ by Proposition~\ref{p:mixture}, and the claim
  in the dissociated case, we obtain
  \begin{equation}
    P[|Y|=\alpha ]=E\big[P[|Y|=\alpha|\mathcal E_Y]\big]=1,
  \end{equation}
  which completes the proof.
\end{proof}

\begin{remark}
  \label{r:Dstar}
  For the rest of the paper, it will be suitable to extend the definition
  of $|y|$ to all possible $y\in \mathbb S$. For those $y\in \mathbb S$ for
  which some of the limits $t^{y,n}$ do not exist, we define
  $|y|=\partial$, where $\partial\notin\mathcal P(\mathbb S)$ is an
  arbitrary symbol. In addition, for $y$ such that $|y|$ exists but is not
  in $\mathcal D_S$, we set $|y|=\partial$ as well.  By
  Proposition~\ref{p:dislimits}, we can then view
  $|y|$ as a map from $\mathbb S$ to
  $\mathcal D^\star_{S}:=\mathcal D_{S}\cup \{\partial\}$.
\end{remark}

As can be seen from
the previous results, the differences between exchangeable and weakly
exchangeable arrays are mostly a matter of notation. That is why, from
now on, we mostly focus on the exchangeable case; the corresponding
statements for the weakly exchangeable case can be derived easily.

\subsection{Relation to exchangeable graphs and  graph limits}

The above construction is a straightforward generalisation of the graph
limit construction from the theory of dense random graphs, which we recall
briefly.

A (vertex) \mindex{exchangeable random graph} is a random graph $G$ with
countably many vertices labelled by $\mathbb N$ whose distribution is
invariant under permutations of the labels. By considering the adjacency
matrix $(G_{ij})_{i,j\in \mathbb N}$ of this graph, it can be viewed as a
$\{0,1\}$-valued weakly exchangeable array whose diagonal entries are $0$.

Graph limits were introduced by Lovász and Szegedy \cite{LS06} while
studying sequences of dense graphs. They encode the limiting density of
finite subgraphs in an infinite graph. Formally, let $\mathcal G_n$ be
the set of all graphs with $n$ vertices labelled by $[n]$. For $m\ge n$
and $F\in \mathcal G_n$ and $G\in \mathcal G_m$, let $\ind(F,G)$ be the
number of injections $\psi\colon[n]\to [m]$ such that
$G_{\psi(i)\psi(j)} = F_{ij}$  for all $i,j\in [n]$. Then, for every
infinite graph $G$ with vertices labelled by $\mathbb N$, one can define
the ``density of $F$ in $G$''
\begin{equation}
  t(F,G)=\lim_{m\to\infty}\frac{\ind(F,G|_{[m]})}{(m)_n}, \qquad F\in
    \mathcal G_n.
\end{equation}

It can be checked easily that $t(\cdot, G)$, restricted to $\mathcal G_n$,
if it exists, is a probability measure on $\mathcal G_n$. This probability
measure, in fact,  coincides with the measure $t^{G,n}$ that was introduced
in \eqref{e:tnweak}, when graphs are identified when their adjacency
matrices.

By construction, every $t(\cdot,G)$ is invariant under action of $\Sigma $,
\begin{equation}
  \label{e:graphex}
  t(F^\pi,G)=t(F,G), \quad\text{for every $F\in \mathcal G_n$,
    $\pi \in \Sigma_n$.}
\end{equation}
Similarly,  the following consistency relation,
corresponding to \eqref{e:consistency} above, holds:
\begin{equation}
  t(F,G)=\sum_{\bar F\in \mathcal G_m: \bar F|_{[n]}= F } t(\bar F, G).
\end{equation}
That means that $(t(F,G))_{F\in \cup_n \mathcal G_n}$, if it exists for
every $F\in \cup_n \mathcal G_n$, can be viewed (again in the sense of
  Kolmogorov's extension theorem) as a distribution of a
random graph, which must be exchangeable due to \eqref{e:graphex}.
This distribution corresponds to $|y|$ of the previous section.

\section{Dynamics of exchangeable arrays}

We now turn to the main goal of this paper, the investigation  of
processes $X=(X(t))_{t\in T}$ taking values in the space $\mathbb S$ of
two-dimensional $S$-valued arrays. Here, $T$ denotes the set of times
which can be both discrete, $T= \mathbb N_0$, or continuous $T=[0,\infty)$.

In the continuous-time case, we assume that the sample paths of $X$ are
càdlàg. Since we endowed $\mathbb S$ with the product topology, this is
the case iff every restriction $X|_{[n]}$ has càdlàg paths in
$\mathbb S_n$, or equivalently, $t\mapsto X_{ij}(t)$ is càdlàg for every
$i,j\in \mathbb N$. We write, $D(S)$ for the space of all càdlàg functions from
$T$ to $S$, endowed with the usual Skorokhod topology. The previous
reasoning implies that $D(\mathbb S)=(D(S))^{\mathbb N^2}$.

By convention, every function on $T$ is càdlàg in the discrete-time
case. This allows us to use the adjective `càdlàg' without specifying
which case we consider.

A $\mathbb S$-valued process $X$ is called \mindex{exchangeable}, if
\begin{equation}
  \label{e:exprocess}
  X^{\boldsymbol \pi}:= (X^{\boldsymbol \pi}(t))_{t\in T} \laweq X,
  \qquad \text{for every $\boldsymbol \pi\in \Sigma^2$.}
\end{equation}
Equivalently, viewing $X$ as an array of functions
$(t\mapsto X_{ij}(t))_{i,j\in \mathbb N}$, it is often useful to regard
$X$ as an exchangeable $D(S)$-valued array. Corresponding to this point of
view, we define an exchangeable $\sigma $-field, $\mathcal E_X$,
associated to the whole process,
\begin{equation}
  \mathcal E_X= \{X^{-1}(A):A \in \mathcal E_{D(S)}\},
\end{equation}
where $\mathcal E_{D(S)}$ is defined as $\mathcal E_S$ with $D(S)$ playing
the role of $S$.

The process $X$ is a \mindex{Markov process} when the Markov property
holds, that is the past $(X(s))_{s\le t}$ and the future $(X(s))_{s\ge t}$
are conditionally independent given the present $X(t)$ for all $t\in T$.
The following proposition gives criteria implying the exchangeability
of a Markov process. Its straightforward proof is left to the reader.

\begin{proposition}
  Let $X$ be an $\mathbb S$-valued Markov process with
  transition probability kernel
  \begin{equation}
    \label{e:transitionexch}
    p_{s,t}(x,A):=P[X(t)\in A \mid X(s)=x], \qquad
    s<t\in T, A\in \mathcal B( \mathbb S).
  \end{equation}
  Then, $X$ is exchangeable if
  \begin{enumerate}
    \item[\textup{(a)}] its initial state $X(0)$ is an $S$-valued exchangeable array, that
    is
    \begin{equation}
      X(0)^{\boldsymbol \pi}\laweq X(0).
    \end{equation}
    \item[\textup{(b)}] its transition kernels
    are invariant under action of $\Sigma^2$, that is for every
    $\boldsymbol\pi\in \Sigma^2$, $s<t\in T$, $x\in \mathbb S$, and
    $A\in \mathcal B(\mathbb S)$
    \begin{equation}
      \label{e:markovkernelexch}
      p_{s,t}(x^{\boldsymbol\pi},A^{\boldsymbol \pi}) = p_{s,t}(x,A).
    \end{equation}
  \end{enumerate}
\end{proposition}

For convenience, we mostly omit ``$\mathbb S$-valued'' from the
terminology and say, e.g., ``exchangeable Markov process'' instead of
``$\mathbb S$-valued exchangeable Markov process''.

We now study how exchangeable Markov processes interact with the
``projection'' operation
$\mathbb S\ni y \mapsto |y|\in \mathcal D^\star_S$,
cf.~Remark~\ref{r:Dstar}. Our first result implies that the projection
of $X(t)$ is in $\mathcal D_S$, a.s., simultaneously for all $t\in T$,
that is one can, a.s., project the process $X$ on the space $\mathcal D_S$
of (the distributions of) dissociated exchangeable arrays,
cf.~Proposition~\ref{p:dislimits}. Remark that Markov property is not
assumed.

\begin{theorem}
  \label{t:markovprojection}
  Let $X$ be an exchangeable process with càdlàg sample paths. Then, $P$-a.s.,
  $|X(t)|\in \mathcal D_S$ for all $t\in T$.
\end{theorem}

\begin{proof}
  In the discrete-time case, it suffices to observe that $X(t)$ is an
  exchangeable $S$-valued array for every $t\in \mathbb N_0$. By
  Proposition~\ref{p:dislimits}, $|X(t)|\in \mathcal D_S$, $P$-a.s., and
  the claim follows, since $\mathbb N_0$ is countable.

  In the continuous-time case, we view $X$ as a $D(S)$-valued
  exchangeable array, cf.~the remark below \eqref{e:exprocess}, and assume
  that this array is dissociated first. Using
  Proposition~\ref{p:dislimits} with $D(S)$ in place of $S$, recalling
  that $|Y|$ there denotes the sequence of limits
  $(t^{Y,n})_{n\in \mathbb N}$, we see that
  for every $n\in \mathbb N$, the sequence $t^{X,n}_m$ of probability
  measures on $D(\mathbb S_n)$ converges weakly as $m\to\infty$ $P$-a.s.~to
  some $t^{X,n}\in\mathcal P(D(\mathbb S_n))$. Moreover, since we assume
  that $X$ is dissociated, $t^{X,n}$ is a.s.~deterministic and coincides
  with the distribution of $X|_{[n]}$, by Proposition~\ref{p:dislimits}(b).

  Let $J_n^X$ be the (deterministic) set of times defined by
  \begin{equation}
    J_n^X=\big\{t\in T:t^{X,n}\big(
        \{x\in D(\mathbb S_n): t \text{ is a jump point of $x$}\}
        \big)>0\big\}.
  \end{equation}
  By the general theory of probability measures on Skorokhod spaces, see
  Chapter 15 in \cite{Bil99}, $J_n^X$ is at most countable.  Therefore,
  using the same argument as in the discrete case, $P$-a.s.,
  $|X(t)|\in \mathcal D_S$ for all $t\in \cup_n J_n^X$.

  For $t\in T\setminus \cup_n J_n^X$, the coordinate projections
  $\phi_t \colon D(\mathbb S_n)\ni x \mapsto x(t) \in \mathbb S_n$ are
  $t^{X,n}$-a.s.~continuous. By Theorem~5.1 of \cite{Bil99}, the weak
  convergence of $t^{X,n}_m$ then implies the existence of the weak limit
  $t^{X(t),n} := \phi_t \circ t^{X,n}
  = \lim_{m\to\infty} \phi_t \circ t^{X,n}_m
  = \lim_{m\to\infty} t^{X(t),n}_m$.
  The limit measures $t^{X(t),n}\in \mathcal P(\mathbb S_n)$ are
  consistent and dissociated, as $t^{X,n}_m$ are, and thus determine a
  probability measure $|X(t)|\in\mathcal D_S$, $P$-a.s., simultaneously
  for all $t\in T\setminus \cup_n J_n^X$.

  The last two paragraphs together imply that for a dissociated $X$,
  $P[|X(t)|\in \mathcal D_S \text{ for all }t\in T ] =1$.

  A general exchangeable càdlàg process $X$ can be written as a mixture of
  dissociated processes by conditioning on $\mathcal E_X$, by
  Proposition~\ref{p:mixture}. Therefore,
  \begin{equation}
    \begin{split}
      P&[|X(t)|\in \mathcal D_S \text{ for all }t\in T ]
      \\&
      = \int_\Omega P[|X(t)|\in \mathcal D_S \text{ for all }t\in T \mid
        \mathcal E_X ](\omega) P(\d \omega ).
    \end{split}
  \end{equation}
  Under $P[\cdot \mid \mathcal E_X]$, the law of $X$ is dissociated, and thus
  the integrand equals~1, a.s., by the previous paragraph. This completes the proof.
\end{proof}

From Proposition~\ref{p:dislimits}, we know that $|X(t)|$ is a regular
conditional distribution of $X(t)$ given its own exchangeable $\sigma $-field
$\mathcal E_{X(t)}$. In general, however, $\mathcal E_{X(t)}$ does not
need to agree with $\mathcal E_X$.  We now show that $|X(t)|$ is also a
regular conditional distribution of $X(t)$ given $\mathcal E_X$.

\begin{lemma}
  \label{l:conditioning}
  \textup{(a)} For every $t\in T$,
  \begin{equation*}
    \mathcal E_{X(t)}  \subset \mathcal E_{X}.
  \end{equation*}

  \textup{(b)} Let $\alpha^X$ be the regular conditional distribution of
  $X$ given $\mathcal E_X$. Then, $P$-a.s.,
  \begin{equation}
    \alpha^X(\omega ,X(t)\in \cdot) = |X(t)|(\omega ,\cdot).
  \end{equation}
  or, equivalently, denoting by $\phi_t$  the projection
  $D(\mathbb S)\ni x \mapsto x(t) \in\mathbb S$,
  \begin{equation}
    \phi_t\circ \alpha^X = |X(t)|.
  \end{equation}
\end{lemma}

\begin{proof}
  (a) Let $B\in \mathcal E_S$. Then $\phi_t^{-1}(B)\in \mathcal E_{D(\mathbb S)}$,
  and thus $X^{-1}(\phi_t^{-1}(B))\in \mathcal E_X$. In addition,
  \begin{equation}
    \begin{split}
      X^{-1}(\phi_t^{-1}(B))
      &= \{\omega \in \Omega : X(\omega )\in \phi_t^{-1}(B)\}
      \\&= \{\omega \in \Omega : (\phi_t \circ X)(\omega )\in B\}
      \\&= \{\omega \in \Omega : X(t)(\omega )\in B\}
      = X(t)^{-1}(B).
    \end{split}
  \end{equation}
  Since, by definition,
  $\mathcal E_{X(t)} = \{X(t)^{-1}(B):  B\in \mathcal E_S\}$, it follows
  that $\mathcal E_{X(t)}\subset \mathcal E_X$, as claimed.

  (b) Heuristically, the proof uses the fact that $|X(t)|$ is a dissociated
  distribution, and dissociated distributions are extremal in the set of
  all exchangeable distributions.

  By properties of regular conditional distributions,  for every
  $C\in \mathcal E_X$, and every bounded measurable
  $f\colon \mathbb S\to \mathbb R$,
  \begin{equation}
    \label{eq:EfX}
    E[\bbone_C f(X(t))]
    = \int_\Omega \bbone_C(\omega ) \alpha^X(\omega ,f\circ \phi_t) P(\d \omega ).
  \end{equation}
  By conditioning on $\mathcal E_{X(t)}$, we obtain
  \begin{equation}
    \begin{split}
      \eqref{eq:EfX}
      =\int_\Omega  P( \d \omega' ) \int_\Omega P( \d \omega
        \mid \mathcal E_{X(t)})(\omega ')
      \bbone_C(\omega )\alpha^X(\omega , f\circ \phi_t).
    \end{split}
  \end{equation}
  Choosing $C\in \mathcal E_{X(t)}\subset \mathcal E_X$ and using  that
  $\bbone_C(\omega )= \bbone_C(\omega ')$,
  $P(\cdot \mid \mathcal E_{X(t)})(\omega ')$-a.s., in this case, we get
  \begin{equation}
    \begin{split}
      \eqref{eq:EfX}
      =\int_\Omega  P(\d \omega ') \bbone_C(\omega') \int_\Omega P( \d \omega
        \mid \mathcal E_{X(t)})(\omega ')
      \alpha^X(\omega , f\circ \phi_t),
    \end{split}
  \end{equation}
  Observe that, as function of $\omega '$, the inner integral is
  $\mathcal E_{X(t)}$ measurable. Therefore,
  \begin{equation}
    \int_\Omega P(\d \omega \mid \mathcal E_{X(t)})
    (\phi_t\circ\alpha^X) (\omega )
  \end{equation}
  is a version of regular conditional distribution of $X(t)$ given
  $\mathcal E_{X(t)}$, that is it equals $|X(t)|$, $P$-a.s. However,
  $|X(t)|$ is dissociated, and thus extremal in the set of all
  exchangeable probability distributions. Therefore, necessarily,
  $(\phi_t\circ \alpha^X)(\omega )=|X(t)|(\omega )$ must hold true for
  $P(\cdot \mid \mathcal E_{X(t)})$-a.e.~$\omega $. This then implies that
  $\phi_t\circ \alpha^X = |X(t)|$, $P$-a.s., as claimed.
\end{proof}

\begin{theorem}
  \label{t:cadlag}
  Let $X$ be an exchangeable process with càdlàg sample paths. Then,
  the projection $|X|= (|X(t)|)_{t\in T}$ has $P$-a.s.~càdlàg sample paths.
\end{theorem}
\begin{proof}
  Assume first that the distribution of $X$ is dissociated, that is
  $\mathcal E_X$ is $P$-trivial. Then, $|X(t)|$ exists $P$-a.s.~simultaneously
  for all $t\in T$, and when it exists, it coincides with the
  distribution of $X(t)$. Since the trajectories of $X$ are càdlàg,
  $\lim_{s\downarrow t}X(s) = X(t)$ pointwise and thus in distribution,
  implying $|X(t)|$ is right continuous.

  On the other hand, let
  $X^-_{ij}(t) = \lim_{s\uparrow t} X_{ij}(t)$. Then, $X^-(t)$ is
  exchangeable array, and, by the same arguments as in the proof of
  Theorem~\ref{t:markovprojection}, one can show that $|X^-(t)|$ exists
  a.s.~simultaneously for all $t\in T$. As the distribution of $X$ is dissociated,
  $|X^-(t)|$ agrees a.s.~with the distribution of $X^-(t)$.
  Repeating the argument from the first part of the proof, we obtain
   $\lim_{s\uparrow t} |X(s)|  = |X^-(t)|$.

  For a general exchangeable process $X$, we write its distribution as a
  mixture of dissociated distributions by conditioning on $\mathcal E_X$,
  \begin{equation}
    P(X\in \cdot) = \int_\Omega  P(\d \omega ) P(X\in \cdot \mid \mathcal E_X)(\omega ).
  \end{equation}
  Under $P(\cdot \mid \mathcal E_X)$, the distribution of $X$ is dissociated,
  and $|X(t)|$ agrees with regular conditional distribution of $X(t)$ given $\mathcal E_X$, by
  Lemma~\ref{l:conditioning}.
  So, by the previous reasoning,
  $P(t\mapsto |X(t)| \text{ is càdlàg} \mid \mathcal E_X)=1$ a.s., and the
  claim follows.
\end{proof}

\begin{theorem}
  Let $X$ be an exchangeable Markov process with càdlàg sample paths.
  Then, $|X|$ is a $\mathcal D^\star_S$-valued Markov process with
  a.s.~càdlàg sample paths.
\end{theorem}

\begin{proof}
  The càdlàg property follows from Theorem~\ref{t:cadlag}. We thus need
  to show that the Markov property is preserved by the map
  $\mathbb S \ni y\mapsto |y|\in \mathcal D_S^\star$. To this end, it is
  sufficient to show that for every $s<t$, and $A\subset \mathcal D_S$
  measurable
  \begin{equation}
    p_{s,t}(x,A^{\leftarrow})=p_{s,t}(x',A^{\leftarrow}), \qquad
    \text{for every $x,x'\in \mathbb S$ with $|x|=|x'|$,}
  \end{equation}
  where $A^\leftarrow = \{x\in \mathbb S: |x|\in A\}$.

  To prove this, observe first
  $A^\leftarrow = (A^\leftarrow)^{\boldsymbol \pi}$ and thus, by the
  exchangeability \eqref{e:transitionexch} of the transition kernel
  \begin{equation}
    p_{s,t}(x,A^\leftarrow)
    = p_{s,t}(x^{\boldsymbol \pi}, (A^\leftarrow)^{\boldsymbol \pi})
    = p_{s,t}(x^{\boldsymbol \pi}, A^\leftarrow),
    \qquad \text{for all } \boldsymbol \pi\in \Sigma^2.
  \end{equation}
  Hence, $x\mapsto p_{s,t}(x,A^\leftarrow)$ is $\mathcal E_S$ measurable.
  In addition, by the same arguments as in Corollary~3.10 of \cite{Ald85},
  $\mathcal E_S$ agrees with the $\sigma $-field generated by the map
  $x\mapsto |x|$, which implies the claim.
\end{proof}

\section{Jumps of discrete-time Markov processes}

In this and the next section, we study in detail the structure of the
jumps of time-homogeneous exchangeable Markov processes. We first
consider processes in discrete time, where the situation is rather simple.

\begin{lemma}
  \label{l:jumpsdiscrete}
  Let $X$ be an exchangeable Markov process in discrete time. Then, the
  array $J_{ij}(t)=\bbone\{X_{ij}(t-1)\neq X_{ij}(t)\}$ encoding its jumps
  at time $t\ge1$
  is also exchangeable. As consequence, only the following two
  possibilities occur a.s.
  \begin{itemize}
    \item $X$ is constant at $t$, that is $X_{ij}(t-1)=X_{ij}(t)$ for all
    $(i,j)\in \mathbb N^2$.
    \item There is a positive proportion of entries which jump, that is
    \begin{equation}
      \lim_{n \to \infty} \frac 1 {n^2}
      \sum_{1\le i,j\le n} \bbone \{X_{ij}(t-1)\neq X_{ij}(t)\} >0.
    \end{equation}
  \end{itemize}
\end{lemma}

\begin{proof}
  The exchangeability of $J_{ij}(t)$ follows directly from the
  exchangeability of $X$. $J(t)$ is then $\{0,1\}$-valued exchangeable
  array. The proportion of ones in this array equals $t^{J,1}(\{1\})$ in
  the notation of Section~\ref{s:excharrays}. Therefore, by
  Proposition~\ref{p:dislimits}, it exists a.s. If
  it is zero, then the array $J_{ij}$ must be identically $0$, a.s.
  Otherwise, there is a positive proportion of entries that jump.
\end{proof}

\subsection{Restrictions of Markov exchangeable processes are not Markov}

If one is interested not only in the occurrence of jumps, but also in
their ``sizes'', this argument can be pushed even further, similarly
as in \cite{Cra16}. For $t\ge 1$, consider $S^2$-valued array
$Z_{ij}:=(X_{ij}(t-1),X_{ij}(t))$, which is again exchangeable. By
Proposition~\ref{p:dislimits} (with $S^2$ in place of $S$), for every
$n\in \mathbb N$, the limit $t^{Z,n}\in \mathcal P(\mathbb  S_n^2)$
exists a.s.

The measure $t^{Z,n}$ can be used to construct a new Markov transition
kernel $q_n$ on $\mathbb S^n$, by disintegrating $t^{Z,n}$
with respect to its first marginal $t^{X(t-1),n}$,
\begin{equation}
  t^{Z,n}(\d y_1, \d y_2) =  t^{X(t-1),n}(\d y_1)
  q^n_{t-1,t}(y_1, \d y_2),
\end{equation}
or, in the case when $S$ is finite, simply by defining
\begin{equation}
  q^n_{t-1,t}(y_1,y_2)=\frac{t^{Z,n}(\{(y_1,y_2)\})}{t^{X(t-1),n}(\{y_1\})},
  \qquad y_1,y_2\in \mathbb S_n,
\end{equation}
(and $q^n_{t-1,t}(y_1,y_2)=\delta_{y_1,y_2}$ in the case when
  $t^{X(t-1),n}(\{y_1\})=0$).
Since, by Proposition~\ref{p:dislimits}, $t^{X,n}$ agrees with the
distribution of $X|_{[n]}$ given $\mathcal E_X$, it is tempting to
interpret the kernels $q_n$ as transition kernels of $X|_{[n]}$ (at least
  conditionally on $\mathcal E_X$), as is done in \cite{Cra16}:
Proposition~4.8 of \cite{Cra16} contains, among others, the following
claim (stated in the notation of the present paper):
\begin{quote}
    Let $X=(X_t)_{t\in T}$ be a time-homogeneous exchangeable Markov
    process, with $T$ being finite.
    Conditioned on $\mathcal E_{X}$, $X$ is dissociated, and, moreover, for
    every $n\in \mathbb N$, the
    restriction $X|_{[n]}$ of $X$ to $\mathbb S_n$ is (conditionally)
    a time-inhomogeneous Markov chain with transition probabilities
    $q^n_{t-1,t}$.
\end{quote}

We now provide a counterexample for a part of this claim, namely that
$X|_{[n]}$ is (conditionally) Markov. We will also see that the
transition kernel of $X|_{[n]}$ is not $q^n$.

\begin{example}
  \label{ex:counterexample}
  We work in the setting of exchangeable random graphs, similarly as
  in~\cite{Cra16}. That is $X_{ij}(t)$ denotes the adjacency matrix of a
  random exchangeable graph, which can thus be viewed as $\{0,1\}$-valued
  weakly exchangeable array with zeros on the diagonal. We fix
  $T=\{0,1,\dots,N\}$ for a large~$N$.

  To construct the process, let $\xi_i$, $i\in \mathbb N$, be
  i.i.d.~Bernoulli($\tfrac 12$) random variables. In the initial
  configuration $X(0)$, we draw an edge between vertices $i\neq j$ (i.e., we
    set $X_{ij}(0)=1$) with probability $p_{ij}(\xi )$, where
  \begin{equation}
    p_{ij}(\xi)= \begin{cases}
      \tfrac14,\qquad &\text{if }\xi_i=\xi_j=0,\\
      \tfrac12,\qquad &\text{if }\xi_i\neq\xi_j,\\
      \tfrac34,\qquad &\text{if }\xi_i=\xi_j=1.
    \end{cases}
  \end{equation}
  All edges are drawn independently.

  To define the dynamics, for every $x\in \mathbb S$, we define
  \begin{equation}
    \label{e:xis}
    \xi_i(x) = \bbone\Big\{\limsup_{n \to \infty} \frac 1n \sum_{j=1}^n x_{ij}>\tfrac
      12\Big\}.
  \end{equation}
  Given the configuration of $X$ at time $t$, we construct $X(t+1)$ as
  follows
  \begin{itemize}
    \item If $\xi_i(X(t)) = \xi_j(X(t)) = 0$, then $X_{ij}$ does not change,
    that is $X_{ij}(t+1)=X_{ij}(t)$.
    \item Otherwise, $X_{ij}$ is refreshed according to $p_{ij}(X(t))$,
    that is $X_{ij}(t)$ is a Bernoulli($p_{ij}(\xi (X(t))$) random
    variable, chosen independently of all other $X_{ij}(t)$'s.
  \end{itemize}

  It is easy to see that the process $X$ is weakly exchangeable. And, by
  construction, it is obviously Markov. In addition, the law of large
  numbers implies that $\xi_i (X(0))=\xi_i$ a.s., and thus $X(1)$, and
  inductively also $X(t)$, $t\ge 1$, have the same distribution as $X(0)$.

  The exchangeable $\sigma $-field $\mathcal E_X$ is $P$-trivial in this
  example, since $X$ is dissociated by construction. Hence, conditioning on
  $\mathcal E_X$ does not have any effect.

  On the other hand, the functions $\xi_i(X(t)) $ cannot be determined
  from any finite restriction
  $X(t)|_{[n]}$. That is, $\xi$'s are ``hidden variables'' for the restriction
  $X|_{[n]}$, and while
  conditionally on $\xi $, $X|_{[n]}$ is Markov, it is not Markov
  unconditionally.

  To prove this, fix $n=2$, that is consider only the state
  of the edge connecting the vertices $1$ and $2$. Then, by an easy
  computation taking into account all possible values of $\xi_1$ and
  $\xi_2$, we obtain that
  $P(X_{12}(t+1)=1 \mid X_{12}(t)=1) = \tfrac{21}{32}$. On the other hand,
  $P(X_{12}(N)=1 \mid X_{12}(t)=1, \forall t<N)$  can be made arbitrarily close
  to one by choosing $N$ large, because if we know that
  $X_{12}(t)=1$ for all $t<N$, then very likely $\xi_1=\xi_2=0$ and thus
  $X_{12}$ never flips:
  \begin{equation}
    \begin{split}
      P(X_{12}(N)=1 \mid {}&X_{12}(t)=1, \forall t<N) =
      \frac
      {P(X_{12}(N)=1, \forall t\le N) }
      {P(X_{12}(N)=1, \forall t< N) }
     \\& = \frac{\frac 14\cdot \frac 14 \cdot 1
        + \frac 12 \cdot (\frac 12)^N+\frac 14 \cdot (\frac 34)^{N}}
      {\frac 14\cdot \frac 14 \cdot 1 + \frac 12 \cdot (\frac 12)^{N-1}
        +\frac 14 \cdot (\frac 34)^{N-1}}
      \xrightarrow{N\to\infty}1.
    \end{split}
  \end{equation}

  This implies that $X_{12}$ is not Markov.
\end{example}

\begin{remark}
  \begin{enumerate}[(a)]
    \item On the technical level, the problem in \cite{Cra16} lies in the fact that
    the relation~(14) therein, which gives certain consistency for the
    kernels $q_n$, does not hold true, in general. This can hinder the
    Markov property of the finite restrictions as shown in
    Example~\ref{ex:counterexample}.

    \item However, in Section~\ref{sec:feller-property} (see
      Theorem~\ref{thm:feller-equiv-locality}), we show that under the
    additional assumption that the ``global'' Markov process $X$ has the
    Feller property (cf., Definition~\ref{def:feller-property}), all the
    ``local'' restrictions $X\vert_{[n]}$ are indeed Markov (and Feller).
    See also Remark~\ref{rem:consitency-equiv-feller}.
  \end{enumerate}
\end{remark}

\section{Jumps of continuous-time Markov processes}

We now study exchangeable Markov processes in continuous time. Similarly
as in discrete time (see Lemma~\ref{l:jumpsdiscrete}), we describe the
possible jumps of this process. The structure here is richer, because the
process is indexed by an uncountable set of times. So, certain events which
have probability $0$ in the discrete settings can occur.

\begin{theorem}
  \label{t:contjumps}
  Let $X$ be exchangeable Markov process with càdlàg paths in continuous
  time, and let $J\subset (0,\infty)$ be the (random) set of times when
  $t \mapsto X_t$ is discontinuous. Then, a.s., $J$ can be written as a disjoint
  union $J=J^1 \dotcup J^2 \dotcup J^3$, where
  \begin{itemize}
    \item $J^1$ is the set of times, where a positive proportion of
    entries of $X$ jumps,
    \begin{equation}
      J^1:=\Big\{t>0: \lim_{n\to\infty} \frac 1 {n^2}
        \sum_{1\le i,j\le n} \bbone\{X_{ij}^-(t)\neq X_{ij}(t)\}>0\Big\},
    \end{equation}
    \item $J^2$ is the set of times, where a positive proportion of entries in one
    row or column of $X$ jumps, $J^2= J^{2,c}\dotcup J^{2,r}$ with
    \begin{align}
      &\begin{aligned}
        J^{2,r}=\Big\{t>0:{}& \text{$\exists! i\in \mathbb N$ s.t.
            $X_{i'j}^-(t)=X_{i'j}(t), \forall i'\neq i,
            j\in \mathbb N$,}
        \\&\text{and $\lim_{n\to\infty} \frac 1n \sum_{j=1}^n
            \bbone\{X_{ij}^-(t)\neq X_{ij}(t)\}>0\Big\}$},\\
      \end{aligned}
      \\&
      \begin{aligned}
        J^{2,c}=\Big\{t>0:{}& \text{$\exists!j\in \mathbb N$ s.t.\!
            $X_{ij'}^-(t)=X_{ij'}(t), \forall j'\neq j, i\in \mathbb N$,}
        \\ &\text{and $\lim_{n\to\infty} \frac 1n \sum_{i=1}^n
            \bbone\{X_{ij}^-(t)\neq X_{ij}(t)\}>0\Big\}$},
      \end{aligned}
    \end{align}
    \item $J^3$ is the set of times, where a unique entry jumps,
    \begin{equation}
    \begin{aligned}
      J^3 = \big\{t>0:{}& \text{$\exists ! i,j\in \mathbb N$ s.t.
          $X_{ij}$ is
          discontinuous at $t$, and }\\
        &\text{$X_{i'j'}$ is continuous at $t$,
          $\forall (i',j')\neq (i,j)$}\big\}.
    \end{aligned}
    \end{equation}
  \end{itemize}
\end{theorem}

\begin{proof}
  We follow the proof of Theorem~3.6 of~\cite{Cra16}.
  By conditioning on $\mathcal E_X$, we may assume without loss of
  generality that $X$ is dissociated.%
  \footnote{Remark that the conditioning on $\mathcal E_X$ effectively
    removes the Markov property of $X$, since $\mathcal E_X$ contains
    information about the whole trajectory. In particular, information
    about certain jumps of $X$ is contained in $\mathcal E_X$.}
  Similarly as in the proof of
  Theorem~\ref{t:markovprojection}, let $J^X$ be the set of times when $X$
  jumps with positive probability,
  \begin{equation}
    J^X:= \{t>0: P(X \text{ is discontinuous at }t)>0\}.
  \end{equation}
  As we remarked previously, this set is at most countable. Hence, by
  considering the arrays $(\bbone \{X_{ij}^-(t)\neq X_{ij}(t)\})_{ij}$,
  $t\in J^X$, and
  using the same arguments as  in the proof of
  Lemma~\ref{l:jumpsdiscrete}, we obtain that $J^X\subset J^1$.

  We now consider times $t\in (0,\infty)\setminus J^X =: C^X$. We first claim
  that for such $t$, the proportion of entries that jump must be 0, that
  is $J^X \supset J^1$.
  Indeed, since $X$ is dissociated, by
  Proposition~\ref{p:dislimits}, for every $t\in (0,\infty)$,
  \begin{equation}
    P\big(X^-_{12}(t)\neq X_{12}(t)\big)
    = \lim_{n\to\infty} \frac 1
    {n^2} \sum_{1\le i,j\le n} \bbone\big\{X^-_{ij}(t)\neq X_{ij}(t)
      \big\},
  \end{equation}
  so if the right-hand side is positive, so must be the left-hand side,
  implying $t\in J^X$.

  We further claim that at $t\in C^X$ it is a.s.~impossible that two
  entries that are not in the same row or in the same column jump at the
  same time. To see this, fix $i\neq k$ and $j\neq l$ and write $J_{ij}$
  for the set of times when $X_{ij}$ jumps. Then,
  \begin{equation}
    P(J_{ij}\cap J_{kl}\cap C^X \neq \emptyset) =
    E\big[ P( J_{ij}\cap J_{kl}\cap C^X \neq \emptyset \mid J_{kl})\big].
  \end{equation}
  The set $J_{kl}$ is at most countable and $J_{ij}$ and $J_{kl}$ are
  independent because $X$ is dissociated. Therefore, the conditional
  probability in the last formula satisfies
  \begin{equation}
    \begin{split}
      P( J_{ij} \cap J_{kl}\cap C^X \neq \emptyset \mid J_{kl}) &\le
      P( J_{ij}\cap C^X \neq \emptyset \mid J_{kl})
      \\&
      = P( J_{ij}\cap C^X \neq \emptyset) = 0,
    \end{split}
  \end{equation}
  where the last equality follows from the definition of $C^X$. This
  yields the claim.

  It remains to be shown that if $X_{ij}$ jumps at $t$, then either it is
  the only entry that jumps, or that there is a positive proportion of
  entries that jump in $i$-th row or $j$-th column. To see this, it is
  sufficient to observe that $X_{i\cdot}:=(X_{ij})_{j\in \mathbb N}$ is
  an exchangeable $D(S)$-valued sequence. In general, $X_{i\cdot}$ is not
  Markov, but we do not need it to be. By conditioning on its
  exchangeable field $\mathcal E_i$ (which, in general, is not related to
    $\mathcal E_X$), $X_{i\cdot}$ is becomes an i.i.d.~sequence, by de
  Finetti's theorem. We may then repeat the arguments of the previous
  paragraphs (applied to sequences instead of arrays) to show that
  $J^{2,r}=\dotcup_{i\in \mathbb N} J^{2,r}_i$, where $J^{2,r}_i$ is the set
  of times when the row $i$ jumps with a positive probability,
  \begin{equation}
    J^{2,r}_i=
    \{t>0:P(X_{i\cdot}^-(t)\neq X_{i\cdot}(t) \mid \mathcal E_i)>0\}.
  \end{equation}
  and, out of $J^{2,r}_i$ there are no simultaneous jumps of two entries
  $X_{ij}$ and $X_{ij'}$ with $j\neq j'$, that is
  \begin{equation}
    P((J_{ij}\cap J_{ij'})\setminus J^{2,r}_i\neq \emptyset \mid \mathcal E_i)
    =0.
  \end{equation}
  This completes the proof.
\end{proof}

Inspection of the previous proof allows us to deduce the following claim about
the discontinuities of the projection $|X|$.

\begin{corollary}
  Let $J^{|X|}$ be the set of times when $t\mapsto |X(t)|$ is
  discontinuous. Then, $J^{|X|}\subset J^1$, where $J^1$ is as in
  \textup{Theorem~\ref{t:contjumps}}.
\end{corollary}

The inclusion in the previous theorem might be strict. As an example,
consider the process started from $X_{ij}(0)$ being
i.i.d.~Bernoulli($\frac 12$), where all entries are refreshed
simultaneously by an independent
i.i.d.~Bernoulli($\frac 12$) array at times of jump of a standard Poisson
process $N_t$. In this case, for every $t\ge 0$, $|X(t)|$ is the
distribution of the i.i.d.~Bernoulli($\frac 12$) array, that is
$J^{|X|}=\emptyset$. On the other hand, $J^1$ agrees with the set of
jumps of $N_t$.

\section{The Feller property}
\label{sec:feller-property}

In the last part of this paper, we discuss the conditions under which
exchangeable $\mathbb S$-valued Markov processes in continuous time have
the Feller property.

Recall the following.

\begin{definition}
  \label{def:feller-property}
  A time-homogeneous $S$-valued Markov process with transition
  kernels $p_t(\cdot,\cdot)$ is called \emph{Feller} if
  \begin{enumerate}[(a)]
    \item For every $g\in C_b(S)$, $t\ge 0$ and  $y\in S$, the map
    $x\mapsto \int g(y)p_t(x,\d y)$ is continuous.

    \item For every $x\in S$ and $g\in C_b(S)$,
    \begin{equation}
      \lim_{t\downarrow 0} \int g(y) p_t(x,\d y) = g(x).
    \end{equation}
  \end{enumerate}
\end{definition}

It is easy to construct exchangeable Markov processes that are not
Feller. E.g., the process considered in Example~\ref{ex:counterexample}
does not satisfy (a) of the Feller property. To see this, take
$g(y)=y_{12}$, $y\in \mathbb S$, and observe that for every $t>0$ there
is $\varepsilon_t>0$ such that if $x_{12}=1$, then
\begin{equation}
  \int g(y) p_t(x,\d y)
  \begin{cases}
    = 1,&\text{if $\xi_1(x)=0$ and $\xi_2(x)=0$,}\\
    < 1-\varepsilon_t, \quad &\text{otherwise.}
  \end{cases}
\end{equation}
Inspecting, the definition \eqref{e:xis} of $\xi_i(x)$, it is easy to see
that it is not continuous function of $x$, and thus $X(t)$ is not Feller.

This example indicates one possibility of how the Feller property can be
violated by exchangeable Markov processes: If the transition kernel
depends on ``non-local exchangeable quantities'', then the process is not
Feller. We now show that this is essentially the only mechanism, how the
Feller property can be violated.

The following definition imposes a very strong ``locality'' of the
distribution of $X$.

\begin{definition}
  An exchangeable Markov process $X$ is called \mindex{consistent} if its
  every restriction $X|_{[n]}$ to $\mathbb S_n$ is Markov with respect to
  its own natural filtration.
\end{definition}

\begin{theorem}
  \label{thm:feller-equiv-locality}
  For a time-homogeneous exchangeable Markov process $X$, the following are
  equivalent:
  \begin{itemize}
    \item[\textup{(i)}] $X$ is consistent and every $X|_{[n]}$ is a Feller process on
    $\mathbb S_n$.
    \item[\textup{(ii)}] $X$ is Feller.
  \end{itemize}
\end{theorem}

\begin{proof}
  We begin with the following observation. Let
  $\mathcal B_n \subset \mathcal B(\mathbb S)$ be the $\sigma $-field
  generated by the canonical projection from $\mathbb S$ to $\mathbb S_n$.
  Then, $X$ is consistent iff its transition kernel satisfies
  \begin{equation}
    \label{e:Mconsistency}
    x\mapsto p_t(x,A) \quad\text{is $\mathcal B_n$-measurable for every
      $A\in \mathcal B_n$ and $t\ge 0$,}
  \end{equation}
  or, in the case when $S$ is finite,
  \begin{equation}
    P[X|_{[n]}(t) =y \mid X(0)=x]=
    P[X|_{[n]}(t) =y \mid X(0)=x']
  \end{equation}
  for every $t\ge 0$, $y\in \mathbb S_n$ and $x,x'\in \mathbb S$ such that
  $x|_{[n]}=x'_{[n]}$.

  (i)$\implies$(ii): Part (a) of the Feller property is equivalent to
  $x\mapsto p_t(x,\cdot)$ is continuous when $\mathcal P(\mathbb S)$ is
  endowed with the topology of weak convergence. As remarked in the
  introduction, this is equivalent to $x\mapsto \int g(y)p_t(x,\d y)$ is
  continuous for every $t\ge 0$ and every cylinder function~$g$, that is
  for every $\mathcal B_n$-measurable $g$, $n\ge 1$. Since the restriction
  $X|_{[n]}$ is Markov by assumption, denoting by $p^n_t(\cdot,\cdot)$ its
  transition kernel, for $g\in \mathcal B_n$,
  \begin{equation}
    \int_{\mathbb S} g(y)p_t(x,\d y)
    = \int_{\mathbb S_n} g(y)p^n_t(x|_{[n]},\d y)
  \end{equation}
  Since $X|_{[n]}$ is assumed to be Feller, the right-hand side is
  a continuous function of $x|_{[n]}$. Since
  $x^k\to x$ in $\mathbb S$ implies $x^k|_{[n]}\to x|_{[n]}$ in
  $\mathbb S_n$, the continuity of the left-hand side follows.

  (ii)$\implies$(i): We first show that $X$ is consistent by showing that
  it satisfies \eqref{e:Mconsistency}. This is equivalent to
  \begin{equation}
    \label{e:consb}
    \int g(y) p_t(x,\d y) =
    \int g(y) p_t(x',\d y)
  \end{equation}
  for all $x'$ with $x'|_{[n]}=x|_{[n]}$ and for all bounded continuous
  $\mathcal B_n$-measurable functions $g$.
  $X$ is assumed to be Feller, so the right-hand side is continuous in
  $x'$, so it is sufficient to verify \eqref{e:consb} for a dense set of
  $x'$ satisfying $x'|_{[n]}=x|_{[n]}$.

  To this end, we use the exchangeability \eqref{e:markovkernelexch}.
  Let
  \begin{equation}
    \Sigma_{(n)}=\{\pi \in \Sigma : \pi(i)=i, \forall i\le n\}
  \end{equation}
  be the set of permutations of $\mathbb N$ that coincide with the
  identity on $[n]$. Let $g$ be a bounded continuous $\mathcal B_n$-measurable
  function. Then, $g(y)= g(y^{\boldsymbol \pi})$ for every
  $y\in \mathbb S$ and $\boldsymbol \pi\in \Sigma_{(n)}^2$. Therefore,
  \eqref{e:markovkernelexch} implies that
  \begin{equation}
    \int g(y) p_t(x,\d y) =
    \int g(y) p_t(x^{\boldsymbol \pi},\d y), \quad
    \text{for every $\boldsymbol \pi\in \Sigma_{(n)}^2$}.
  \end{equation}
  Hence, to prove \eqref{e:consb} it suffices show that there is
  $\bar y\in \mathbb S$ with $\bar y|_{[n]}= x|_{[n]}$ such that the set
  $\{\bar y^{\boldsymbol \pi}:\boldsymbol \pi\in \Sigma_{(n)}^2\}$ is
  dense in $\{x':x'|_{[n]}= x|_{[n]}\}$.

  We construct such $\bar y$ by picking it randomly. To this end, let
  $U\subset S$ be a countable dense subset of $S$, and let $\rho $ be a
  probability measure on $U$ such that $\rho ({x})>0$ for every $x\in U$.
  Let $Y$ be a $\mathbb S$-valued random variable on some auxiliary
  probability space $(\tilde \Omega ,\tilde P)$ such that
  $Y|_{[n]}=x|_{[n]}$, $\tilde P$-a.s., and $Y_{ij}$, $i>n$ or $j>n$, are
  i.i.d.~$\rho $-distributed.

  Then, for $m\ge n$ and $z\in \mathbb S_m$ such that $z|_{[n]}=x_{[n]}$ and
  $z_{ij} \in U$ for $(i,j)\in [m]^2\setminus [n]^2$, we have
  $\tilde P (Y|_{[m]}=z)>0$. So, by the 0-1 law, there is $\tilde P$-a.s.
  $\boldsymbol \pi \in \Sigma_{(n)}^2$ such that
  $Y^{\boldsymbol \pi}|_{[m]}=z$. This implies that
  $\{Y^{\boldsymbol \pi}:\pi\in \Sigma_{(n)}^2\}$ is dense in
  $\{x':x'|_{[n]}= x|_{[n]}\}$, $P$-a.s., which is more than sufficient
  for the existence of $\bar y$ of the last paragraph. This completes the
  proof that $X$ is consistent.

  The consistency then implies that $X|_{[n]}$ is Markov with respect to
  its natural filtration for every $n\ge 1$. The Feller property of $X|_{[n]}$
  is then a direct consequence of the Feller property of $X$.
\end{proof}

\begin{remark}
\label{rem:consitency-equiv-feller}
  If $S$ is finite, then $\mathbb S_n$ is finite as well. Every càdlàg
  Markov process on a finite state space is Feller. Therefore, in this case, the
  consistency of $X$ is equivalent to Feller property. This was proved in
  the exchangeable random graph case in \cite{Cra17}.
\end{remark}

\bibliographystyle{plain}

\def\cprime{$'$}

\end{document}